\documentclass[12pt]{amsart}
\usepackage{amsmath,amsfonts,amssymb,amsthm,amsbsy,amscd,latexsym,cite}
\usepackage[cp1251]{inputenc}
\usepackage[english]{babel}
\usepackage[%
  colorlinks,
  breaklinks,
  linkcolor=blue,
  citecolor=blue,
  menucolor=blue,
  pagecolor=blue,
]{hyperref}

\multlinegap=\parindent

\begin{document}

\title[On~the~approximability by~root classes of~groups]{On~conditions for~the~approximability of~the~fundamental groups of~graphs of~groups by~root classes of~groups}

\author{E.~V.~Sokolov}
\address{Ivanovo State University, Russia}
\email{ev-sokolov@yandex.ru}

\begin{abstract}
Suppose that $\Gamma$ is a~non-empty connected graph, $\mathfrak{G}$ is the~fundamental group of~a~graph of~groups over~$\Gamma$, and~$\mathcal{C}$ is a~root class of~groups (the~last means that $\mathcal{C}$ contains non-trivial groups and~is closed under~taking subgroups, extensions, and~Cartesian powers of~a~certain type). It~is known that $\mathfrak{G}$ is residually a~$\mathcal{C}$\nobreakdash-group if it has a~homomorphism onto~a~group from~$\mathcal{C}$ acting injectively on~all vertex groups. We prove that, in~this assertion, the~words ``vertex groups'' can be replaced by~``edge subgroups'' provided all vertex groups are residually $\mathcal{C}$\nobreakdash-groups. We also show that the~converse doesn't need to~hold if $\mathcal{C}$ consists of~periodic groups and~contains at~least one infinite group.
\end{abstract}

\keywords{Residual finiteness, residual $p$\nobreakdash-finiteness, residual solvability, approximability by~root classes, generalized free product, HNN-extension, tree product, fundamental group of~a~graph of~groups}

\thanks{The~study was supported by~the~Russian Science Foundation grant No.~22-21-00166,\\ \url{https://rscf.ru/en/project/22-21-00166/}}

\maketitle

\newtheorem{theorem}{Theorem}
\newtheorem{proposition}{Proposition}[section]

\section{Introduction. Statement of~results}

This article continues the~paper~\cite{SokolovTumanova2020LJM}, where the~relationship is considered between two properties of~the~fundamental group of~a~graph of~groups. The~first of~these properties is the~approximability by~a~root class~$\mathcal{C}$, and~the~second is the~existence of~a~homomorphism of~the~fundamental group that maps it onto~a~group from~$\mathcal{C}$ and~acts injectively on~all edge subgroups.

Let us recall that a~group~$X$ is~said to~be \emph{approximable by~a~class of~groups~$\mathcal{C}$} (this term was introduced by~A.~I.~Malcev~\cite{Malcev1949MS}) or~\emph{residually a~$\mathcal{C}$\nobreakdash-group} (this term belongs to~P.~Hall~\cite{Gruenberg1957PLMS}) if, for~each $x \in X \setminus \{1\}$, there exists a~homomorphism~$\sigma$ of~$X$ onto~a~group from~$\mathcal{C}$ (a~\emph{$\mathcal{C}$\nobreakdash-group}) such that $x\sigma \ne 1$. The~property of~\emph{residual finiteness} (i.\,e.,~the~approximability by~the~class of~all finite groups) is most famous because a~finitely presented group with~this property has a~solvable word problem~\cite{Malcev1958PIvPI}. At~the~same time, the~approximability by~other classes of~groups is also considered in~the~literature, and~many of~these classes are root classes of~groups.

In~accordance with~one of~the~equivalent definitions (see Proposition~\ref{p1401} below), a~class of~groups~$\mathcal{C}$ is called a~\emph{root class} if it contains non-trivial groups and~is closed under~taking subgroups, extensions, and~Cartesian products of~the~form $\prod_{y \in Y} X_{y}$, where $X, Y \in \mathcal{C}$ and~$X_{y}$ is an~isomorphic copy of~$X$ for~each $y \in Y$. The~concept of~a~root class was introduced by~K.~Gruenberg~\cite{Gruenberg1957PLMS} and~turned~out to~be very useful in~studying the~approximability of~the~fundamental groups of~various graphs of~groups~\cite{AzarovTieudjo2002PIvSU, Goltsov2015MN, Tumanova2015IVM, Tumanova2019SMJ, SokolovTumanova2020IVM, Sokolov2021SMJ1, Sokolov2021SMJ2, Sokolov2021JA, Sokolov2022CA, Sokolov2023JGT}. Thanks to~its use, it~became possible, in~particular, to~make significant progress in~the~study of~the~residual $p$\nobreakdash-finiteness (where $p$ is a~prime number) and~the~residual solvability of~such groups.{\parfillskip=0pt{}\par}

Everywhere below, it is assumed that $\Gamma = (\mathcal{V}, \mathcal{E})$ is a~non-empty connected undirected graph with~a~vertex set~$\mathcal{V}$ and~an~edge set~$\mathcal{E}$ (loops and~multiple edges are allowed). We define a~directed \emph{graph of~groups}
$$
\mathcal{G}(\Gamma) = \big(\Gamma,\ G_{v}\ (v \in \mathcal{V}),\ H_{e},\ \varphi_{\pm e}\ (e \in \mathcal{E})\big)
$$
over~$\Gamma$ by~assigning to~each vertex $v \in \mathcal{V}$ some group~$G_{v}$, and~to~each edge $e \in \mathcal{E}$ a~direction, a~group~$H_{e}$, and~injective homomorphisms $\varphi_{+e}\colon H_{e} \to G_{e(1)}$, $\varphi_{-e}\colon H_{e} \to G_{e(-1)}$ (where $e(1)$ and~$e(-1)$ denote the~vertices that are the~ends of~$e$). Let us refer to~the~groups~$G_{v}$ and~the~subgroups $H_{+e} = H_{e}\varphi_{+e}$, $H_{-e} = H_{e}\varphi_{-e}$ as~\emph{vertex groups} and~\emph{edge subgroups}, respectively.

For~each maximal tree $\mathcal{T} = (\mathcal{V}, \mathcal{E}_{\mathcal{T}})$ in~$\Gamma$, we can consider a~group representation whose generators are the~generators of~the~groups~$G_{v}$ ($v \in \mathcal{V}$) and~symbols~$t_{e}$ ($e \in \mathcal{E} \setminus \mathcal{E}_{\mathcal{T}}$), and~whose defining relations are the~relations of~$G_{v}$ ($v \in \mathcal{V}$) and~all possible relations of~the~form
\begin{align*}
h\varphi_{+e}^{\vphantom{1}} &= h\varphi_{-e}^{\vphantom{1}}\ 
(e \in \mathcal{E}_{\mathcal{T}}^{\vphantom{1}},\ 
h \in H_{e}^{\vphantom{1}}),\\
t_{e}^{-1}(h\varphi_{+e}^{\vphantom{1}})t_{e}^{\vphantom{1}} 
&= h\varphi_{-e}^{\vphantom{1}}\ 
(e \in \mathcal{E} \setminus \mathcal{E}_{\mathcal{T}}^{\vphantom{1}},\ 
h \in H_{e}^{\vphantom{1}}),
\end{align*}
where $h\varphi_{\varepsilon e}$ ($\varepsilon = \pm 1$) is a~word in~the~generators of~$G_{e(\varepsilon)}$ that represents the~image of~$h$ under~$\varphi_{\varepsilon e}$. It~is known~\cite[\S~5.1]{Serre1980} that, for~all maximal trees in~$\Gamma$, the~corresponding representations of~the~described form determine, up~to~isomorphism, the~same group. This group is called the~\emph{fundamental group} of~the~graph of~groups~$\mathcal{G}(\Gamma)$ and~is usually denoted by~$\pi_{1}(\mathcal{G}(\Gamma))$. It~is also known~\cite[\S~5.2]{Serre1980} that the~identity mappings of~generators determine embeddings of~the~vertex groups~$G_{v}$ ($v \in \mathcal{V}$) into~$\pi_{1}(\mathcal{G}(\Gamma))$, and~therefore these groups can be considered subgroups of~$\pi_{1}(\mathcal{G}(\Gamma))$. Let us write $\pi_{1}(\mathcal{G}(\Gamma), \mathcal{T})$ to~specify the~maximal tree~$\mathcal{T}$ used to~construct a~representation of~$\pi_{1}(\mathcal{G}(\Gamma))$.

An~important role in~the~study of~the~approximability of~$\pi_{1}(\mathcal{G}(\Gamma))$ by~a~root class~$\mathcal{C}$ belongs to~homomorphisms that map this group onto~$\mathcal{C}$\nobreakdash-groups and~act injectively on~all its vertex groups (for~brevity, we call them below \emph{homomorphisms of~type~$(\mathfrak{i})$}). If~such a~homomorphism exists, then $\pi_{1}(\mathcal{G}(\Gamma))$ is residually a~$\mathcal{C}$\nobreakdash-group~\cite[Proposition~7]{Sokolov2021SMJ1}. But~at~the~same time, all vertex groups turn~out to~be $\mathcal{C}$\nobreakdash-groups, and~thus we are dealing with~a~very special case. In~the~general case, when the~vertex groups need~not belong to~$\mathcal{C}$, the~so-called ``filtration approach'' is most often used to~prove that $\pi_{1}(\mathcal{G}(\Gamma))$ is residually a~$\mathcal{C}$\nobreakdash-group. This method was proposed by~G.~Baumslag~\cite{BaumslagG1963TAMS} and~consists in~the~approximability of~$\pi_{1}(\mathcal{G}(\Gamma))$ by~the~fundamental groups of~graphs of~groups having homomorphisms of~type~$(\mathfrak{i})$. For~all its productivity, this approach also has certain limitations, a~detailed discussion of~which is given in~\cite{Sokolov2021SMJ1}. Therefore, the~problem arises of~finding conditions of~a~general form which make it possible to~establish the~approximability of~$\pi_{1}(\mathcal{G}(\Gamma))$ without using the~filtration method. In~the~present article, the~following theorem is proved in~this direction.

\begin{theorem}\label{t2201}
Suppose that $\mathcal{C}$ is a~root class of~groups and~all vertex groups~$G_{v}$ \textup{(}$v \in \mathcal{V}$\textup{)} are residually $\mathcal{C}$\nobreakdash-groups. If~there exists a~homomorphism~$\sigma$ of~$\pi_{1}(\mathcal{G}(\Gamma))$ onto~a~group from~$\mathcal{C}$ acting injectively on~all edge subgroups~$H_{\varepsilon e}$~\textup{(}$e \in \mathcal{E}$\textup{,} $\varepsilon = \pm 1$\textup{),} then $\pi_{1}(\mathcal{G}(\Gamma))$ is residually a~$\mathcal{C}$\nobreakdash-group.
\end{theorem}

The~above theorem generalizes Theorem~1 from~\cite{SokolovTumanova2020LJM}, which says that the~same assertion holds if $\Gamma$ is a~finite graph. It~becomes possible to~discard this finiteness condition due~to~a~completely different proof. The~latter is based on~two facts on~the~structure of~subgroups of~$\pi_{1}(\mathcal{G}(\Gamma))$, which are of~independent interest (see Propositions~\ref{p2202} and~\ref{p2203} below).

The~map~$\sigma$ appearing in~Theorem~\ref{t2201} is referred to~below as~a~\emph{homomorphism of~type~$(\mathfrak{ii})$}. Let us note that, if a~more complex free construction is composed of~several simpler ones and~we want to~prove the~existence of~a~homomorphism of~type~$(\mathfrak{i})$ of~the~complex construction, then it may not be enough for~the~simple constructions to~have homomorphisms of~the~same type. However, if we replace $(\mathfrak{i})$ with~$(\mathfrak{ii})$, the~situation improves: this can be seen, for~example, from~the~proof of~Theorem~2 in~\cite{SokolovTumanova2023SMJ}. Thus, the~search for~conditions for~the~existence of~homomorphisms of~type~$(\mathfrak{ii})$ is an~independent problem, \pagebreak the~results of~which can be used both~for~direct proof of~approximability and~for the~subsequent application of~the~filtration approach.

Now let us turn to~the~question of~whether the~converse of~Theorem~\ref{t2201} holds. The~following assertion is proved in~\cite{SokolovTumanova2020LJM}.

\begin{theorem}\label{t2502}
If $\mathcal{C}$ is a~root class of~groups containing at~least one infinite group and~not containing some \textup{(}absolutely\textup{)} free group of~finite or~countable rank\textup{,} then for~any graph $\Gamma = (\mathcal{V}, \mathcal{E})$\textup{,} there exists a~graph of~groups
$$
\mathcal{G}(\Gamma) = \big(\Gamma,\ G_{v}\ (v \in \mathcal{V}),\ H_{e},\ \varphi_{\pm e}\ (e \in \mathcal{E})\big)
$$
such that\textup{:}

\textup{1)}\hspace{1ex}all groups~$G_{v}$ \textup{(}$v \in \mathcal{V}$\textup{)} are residually $\mathcal{C}$\nobreakdash-groups\textup{;}

\textup{2)}\hspace{1ex}all subgroups~$H_{\varepsilon e}$ \textup{(}$e \in \mathcal{E}$\textup{,} $\varepsilon = \pm 1$\textup{)} belong to~$\mathcal{C}$ and~are non-trivial\textup{;}

\textup{3)}\hspace{1ex}$\pi_{1}(\mathcal{G}(\Gamma))$ is residually a~$\mathcal{C}$\nobreakdash-group\textup{;}

\textup{4)}\hspace{1ex}if $\mathcal{T}$ is a~maximal tree in~$\Gamma$ and~$\sigma$ is a~homomorphism of~$\pi_{1}(\mathcal{G}(\Gamma), \mathcal{T})$ onto~a~group from~$\mathcal{C}$\textup{,} then for~any $e \in \mathcal{E}$\textup{,} $\varepsilon = \pm 1$\textup{,} the~relation $\operatorname{ker}\sigma \cap H_{\varepsilon e} \ne 1$ holds.
\end{theorem}

The~above theorem says that, in~very many cases, the~approximability of~$\pi_{1}(\mathcal{G}(\Gamma))$ by~a~root class~$\mathcal{C}$ does~not imply that this group has a~homomorphism of~type~$(\mathfrak{ii})$. However, its proof given in~\cite{SokolovTumanova2020LJM} makes essential use of~the~fact that the~vertex groups of~the~constructed graph of~groups do~not have homomorphisms onto~$\mathcal{C}$\nobreakdash-groups that are injective on~their edge subgroups. Thus, the~absence of~a~homomorphism of~type~$(\mathfrak{ii})$ is~explained not~so~much by~the~structure of~the~group~$\pi_{1}(\mathcal{G}(\Gamma))$ as~a~whole, but~by~the~properties of~its vertex groups. In~the~present article, we consider the~case when all vertex groups belong to~$\mathcal{C}$ and,~therefore, cannot interfere with~the~existence of~a~homomorphism of~the~indicated type. The~following assertion is proved.

\begin{theorem}\label{t2601}
If $\mathcal{C}$ is a~root class of~groups consisting of~periodic groups and~containing at~least one infinite group\textup{,} then for~any graph $\Gamma = (\mathcal{V}, \mathcal{E})$ with~a~non-empty set of~edges\textup{,} there exists a~graph of~groups
$$
\mathcal{G}(\Gamma) = \big(\Gamma,\ G_{v}\ (v \in \mathcal{V}),\ H_{e},\ \varphi_{\pm e}\ (e \in \mathcal{E})\big)
$$
such that\textup{:}

\textup{1)}\hspace{1ex}all groups~$G_{v}$ \textup{(}$v \in \mathcal{V}$\textup{)} belong to~the~class~$\mathcal{C}$\textup{;}

\textup{2)}\hspace{1ex}all subgroups~$H_{\varepsilon e}$ \textup{(}$e \in \mathcal{E}$\textup{,} $\varepsilon = \pm 1$\textup{)} are non-trivial\textup{;}

\textup{3)}\hspace{1ex}$\pi_{1}(\mathcal{G}(\Gamma))$ is residually a~$\mathcal{C}$\nobreakdash-group\textup{;}

\textup{4)}\hspace{1ex}if $\mathcal{T}$ is a~maximal tree in~$\Gamma$ and~$\sigma$ is a~homomorphism of~$\pi_{1}(\mathcal{G}(\Gamma), \mathcal{T})$ onto~a~group from~$\mathcal{C}$\textup{,} then for~any $e \in \mathcal{E}$\textup{,} $\varepsilon = \pm 1$\textup{,} the~relation $\operatorname{ker}\sigma \cap H_{\varepsilon e} \ne 1$ holds.
\end{theorem}

It remains an~open question: whether an~analog of~Theorem~\ref{t2601} holds if $\mathcal{C}$ consists of~finite groups or~contains non-periodic groups.

\section{On~some free constructions of~groups\\ and~subgroups of~these constructions}

Let us recall that if $\mathcal{E} = \{e\}$ and~$e(1) \ne e(-1)$, then the~fundamental group~$\pi_{1}(\mathcal{G}(\Gamma))$ is~said to~be the~\emph{generalized free product} of~the~groups~$G_{e(1)}$ and~$G_{e(-1)}$ with~the~\emph{amalgamated subgroup} $H_{+e} = H_{-e}$. If~$\mathcal{V} = \{v\}$ and~$\mathcal{E} \ne \varnothing$, then $\pi_{1}(\mathcal{G}(\Gamma))$ is the~\emph{HNN-extension} of~the~group~$G_{v}$ with~the~family of~\emph{stable letters} $\{t_{e} \mid e \in \mathcal{E}\}$ and~$G_{v}$ is the~\emph{base group} of~this HNN-extension. Finally, if $\Gamma$ is a~tree, then $\pi_{1}(\mathcal{G}(\Gamma))$ is called the~\emph{tree product} of~the~groups~$G_{v}$ ($v \in \mathcal{V}$). Below, we also use the~construction of~the~(ordinary) \emph{free product} of~a~family of~groups. Its definition and~properties can be found, for~example, in~\cite[\S~6.2]{Robinson1996}.
\pagebreak

The~next proposition follows from~the~results of~\cite{Cohen1974JAMS}.

\begin{proposition}\label{p1106}
If $\mathcal{V} = \{v\}$\textup{,} $\mathcal{E} \ne \varnothing$\textup{,} and~$N$ is a~normal subgroup of~$\pi_{1}(\mathcal{G}(\Gamma))$ that trivially intersects all edge subgroups~$H_{\varepsilon e}$ \textup{(}$e \in \mathcal{E}$\textup{,} $\varepsilon  = \pm 1$\textup{),} then $N$ splits as~the~\textup{(}ordinary\textup{)} free product of~a~free group and~groups\textup{,} each of~which is isomorphic to~the~subgroup $N \cap G_{v}$.
\end{proposition}

\begin{proposition}\label{p1107}
\textup{\cite{Kurosch1934MA}}
If $\mathbb{P}$ is a~free product of~groups~$A_{\lambda}$ \textup{(}$\lambda \in \Lambda$\textup{)} and~$N$ is a~subgroup of~$\mathbb{P}$\textup{,} then $N$ splits as~the~free product of~a~free group and~groups\textup{,} each of~which is isomorphic to~the~subgroup $x^{-1}Nx \cap A_{\lambda}$ for~some $x \in \mathbb{P}$\textup{,} $\lambda \in \Lambda$.
\end{proposition}

\begin{proposition}\label{p1201}
\textup{\cite[Proposition~1]{Sokolov2021SMJ1}}
Suppose that $\Delta$ is a~non-empty connected subgraph of~$\Gamma$ and~$\mathcal{G}(\Delta)$ is the~graph of~groups whose vertices and~edges are associated with~the~same groups\textup{,} directions\textup{,} and~homomorphisms as~in~$\mathcal{G}(\Gamma)$. If~$\mathcal{T}$ is a~maximal tree in~$\Gamma$ such that $\Delta \cap \mathcal{T}$ is a~maximal tree in~$\Delta$\textup{,} then the~identity mapping of~the~generators of~the~fundamental group $\pi_{1}(\mathcal{G}(\Delta), \Delta \cap \mathcal{T})$ into~$\pi_{1}(\mathcal{G}(\Gamma), \mathcal{T})$ defines an~injective homomorphism.
\end{proposition}

\begin{proposition}\label{p2202}
If $\Gamma$ is a~tree\textup{,} then $\pi_{1}(\mathcal{G}(\Gamma))$ can be embedded in~the~HNN-extension
$$
\mathbb{E} = \big\langle G_{v}^{\vphantom{1}}\ (v \in \mathcal{V}),\ t_{e}^{\vphantom{1}}\ (e \in \mathcal{E});\ 
t_{e}^{-1}h\varphi_{+e}^{\vphantom{1}}t_{e}^{\vphantom{1}} = h\varphi_{-e}^{\vphantom{1}}\ (e \in \mathcal{E},\ h \in H_{e}^{\vphantom{1}}) \big\rangle.
$$
\end{proposition}

\begin{proof}
It is well known that the~normal closure of~the~base group of~an~HNN-extension is a~tree product of~isomorphic copies of~this group. Since the~base group~$\mathbb{P}$ of~the~HNN-extension~$\mathbb{E}$ is the~free product of~the~groups~$G_{v}$ ($v \in \mathcal{V}$), its normal closure turns~out to~be a~``forest'' product (which corresponds in~the~general case to~a~forest, not a~tree) of~isomorphic copies of~the~groups~$G_{v}$ ($v \in \mathcal{V}$). We describe in~more detail how this product is structured and~then indicate a~subtree whose fundamental group is isomorphic to~$\pi_{1}(\mathcal{G}(\Gamma))$.

Let $T$ be a~free group with~basis $\{t_{e} \mid e \in \mathcal{E}\}$. Consider the~graph~$\Gamma^{\prime}$ with~the~set of~vertices $\mathcal{V}^{\prime} = \mathcal{V} \times T$ and~the~set of~edges~$\mathcal{E}^{\prime}$ indexed by~the~set~$\mathcal{E} \times T$ and~defined as~follows: for~any $e \in \mathcal{E}$, $t \in T$, the~edge~$e^{\prime}$ with~index~$(e,t)$ connects the~vertices $e^{\prime}(1) = \big(e(1),\,t_{e}t\big)$ and~$e^{\prime}(-1) = \big(e(-1),\,t\big)$. It~is easy to~see that $\Gamma^{\prime}$ has no multiple edges or~loops. To~prove the~acyclicity of~$\Gamma^{\prime}$, we show that if $L$ is a~simple chain in~this graph which has a~non-zero length and~joins a~vertex $u^{\prime}= (u,r)$ to~a~vertex $w^{\prime}= (w,s)$, then there exists an~element $\tau \in T \setminus \{1\}$ such that $\tau r = s$ and,~therefore, $u^{\prime}\ne w^{\prime}$.

Indeed, let $e^{\prime}$ be the~edge of~$L$ connecting the~vertices $\big(e(1),\,t_{e}t\big)$ and~$\big(e(-1),\,t\big)$ for~some $e \in \mathcal{E}$, $t \in T$. If~the~movement along the~chain (from~$u^{\prime}$ to~$w^{\prime})$ includes a~transition from~$\big(e(1),\,t_{e}t\big)$ to~$\big(e(-1),\,t\big)$, then we associate~$e^{\prime}$ with~the~element~$t_{e^{\prime}}^{\vphantom{1}} = t_{e}^{-1}$, otherwise we put $t_{e^{\prime}}^{\vphantom{1}} = t_{e}^{\vphantom{1}}$. Let us denote by~$\tau$ the~product of~all elements corresponding to~the~edges of~$L$ and~taken in~the~opposite order to~the~movement indicated above. Since $L$ is a~simple chain, its adjacent edges are associated with~elements that are not mutually inverse. Hence, $\tau$ (as~an~element of~$T$) has a~reduced form of~non-zero length and,~therefore, is non-trivial~\cite[\S~2.1]{Robinson1996}. The~equality $\tau r = s$ follows from~the~definition of~the~elements~$t_{e^{\prime}}$.

For any $v \in \mathcal{V}$, $t \in T$, we denote by~$G_{v,t}$ an~isomorphic copy of~$G_{v}$ and~by~$\iota_{v,t}$ the~isomorphism $G_{v} \to G_{v,t}$. Let $\mathcal{G}^{\prime}(\Gamma^{\prime})$ be the~graph of~groups such that the~group~$G_{v,t}$ is~assigned to~the~vertex $(v,t) \in \mathcal{V}^{\prime}$ ($v \in \mathcal{V}$, $t \in T$), while the~group~$H_{e}$ and~the~homomorphisms
$$
\varphi_{+e^{\prime}} = \varphi_{+e}\iota_{e(1),\,t_{e}t},\quad 
\varphi_{-e^{\prime}} = \varphi_{-e}\iota_{e(-1),\,t}
$$
are assigned to~the~edge $e^{\prime} \in \mathcal{E}^{\prime}$ with~index~$(e,t)$ ($e \in \mathcal{E}$, $t \in T$). For~any element $\tau \in T$, we consider the~mapping of~the~vertices of~$\Gamma^{\prime}$ defined by~the~rule $(v,t) \mapsto (v,t\tau)$ and~the~corresponding isomorphisms of~the~vertex groups $\iota_{v,t}^{-1}\iota_{v,t\tau}^{\vphantom{1}}$. It~is easy to~see that these mappings induce automorphisms of~$\Gamma^{\prime}$ and~$\mathcal{G}^{\prime}(\Gamma^{\prime})$, and,~hence, an~automorphism~$\alpha_{\tau}$ of~$\pi_{1}(\mathcal{G}^{\prime}(\Gamma^{\prime}))$. It~is also obvious that $\alpha_{\tau_{1}\tau_{2}} = \alpha_{\tau_{1}}\alpha_{\tau_{2}}$ for~any $\tau_{1}, \tau_{2} \in T$. Therefore, we can consider the~split extension~$S$ of~$\pi_{1}(\mathcal{G}^{\prime}(\Gamma^{\prime}))$ by~$T$ such that the~conjugation by~$\tau \in T$ acts on~$\pi_{1}(\mathcal{G}^{\prime}(\Gamma^{\prime}))$~as~$\alpha_{\tau}$.

The~group~$S$ has the~representation
$$
\left\langle 
\begin{aligned}
&G_{v,t}\ (v \in \mathcal{V},\ t \in T),\\ 
&t_{e}\ (e \in \mathcal{E})
\end{aligned}
\ \left|\ 
\begin{aligned}
\tau^{-1}g\tau &= g\alpha_{\tau}\ (g \in G_{v,t},\ v \in \mathcal{V},\ t,\tau \in T),\\ 
h\varphi_{+e}\iota_{e(1),\,t_{e}t} &= h\varphi_{-e}\iota_{e(-1),\,t}\ (e \in \mathcal{E},\ h \in H_{e},\ t \in T)
\end{aligned}
\right.
\right\rangle.
$$
Since, for~any $v \in \mathcal{V}$, $g \in G_{v}$, $\tau \in T$, the~equalities
$$
\tau^{-1}g\iota_{v,1}^{\vphantom{1}}\tau = g\iota_{v,1}^{\vphantom{1}}\alpha_{\tau} = g\iota_{v,1}^{\vphantom{1}}\iota_{v,1}^{-1}\iota_{v,\tau}^{\vphantom{1}} = g\iota_{v,\tau}^{\vphantom{1}}
$$
hold in~$S$, the~generators of~$G_{v,t}$ ($v \in \mathcal{V}$, $t \in T \setminus \{1\}$) can be excluded from~this representation together with~the~relations
$$
\tau^{-1}g\tau = g\alpha_{\tau}\quad (g \in G_{v,1},\ v \in \mathcal{V},\ \tau \in T).
$$
As~a~result, the~relations
$$
\tau^{-1}g\tau = g\alpha_{\tau}\quad (g \in G_{v,t},\ v \in \mathcal{V},\ t \in T \setminus \{1\},\ \tau \in T)
$$
turn into~identities, the~relations
$$
h\varphi_{+e}\iota_{e(1),\,t_{e}t} = h\varphi_{-e}\iota_{e(-1),\,t}\quad (e \in \mathcal{E},\ h \in H_{e},\ t \in T)
$$
take the~form
$$
t_{e}^{-1}(h\varphi_{+e}^{\vphantom{1}}\iota_{e(1),1}^{\vphantom{1}})t_{e}^{\vphantom{1}} = h\varphi_{-e}^{\vphantom{1}}\iota_{e(-1),1}^{\vphantom{1}}\quad (e \in \mathcal{E},\ h \in H_{e}),
$$
and~the~representation of~$S$ can be turned into~the~representation of~$\mathbb{E}$ by~identifying the~elements of~$G_{v}$ ($v \in \mathcal{V}$) and~their images under~$\iota_{v,1}$. Therefore, $\mathbb{E} \cong S$.

Now let us build an~embedding of~$\Gamma$ into~$\Gamma^{\prime}$. To~do this, we fix some vertex $u \in \mathcal{V}$ and~argue by~induction on~the~length of~a~(unique) path in~the~tree~$\Gamma$ joining an~arbitrarily chosen vertex $v \in \mathcal{V}$~to~$u$.

Each vertex $v \in \mathcal{V}$ will be mapped to~a~vertex of~the~form~$(v,t)$ for~some $t \in T$. Let us associate~$u$ with~$(u,1)$. If~$v \in \mathcal{V}$ is a~vertex other than~$u$, $e$ is the~last edge of~the~path joining~$u$ to~$v$, $\varepsilon = \pm 1$ is the~number satisfying the~equality $v = e(\varepsilon)$, and~the~vertex~$e(-\varepsilon)$ corresponds to~the~vertex $\big(e(-\varepsilon),\, t\big)$ for~some $t \in T$, then we associate the~vertex $v = e(\varepsilon)$ with~$\big(e(\varepsilon),\, t_{e}^{\varepsilon}t\big)$. Let $e^{\prime}$ be the~edge of~$\Gamma^{\prime}$ with~index~$(e,t)$ if~$\varepsilon = 1$, or~with~index~$\big(e,\, t_{e}^{-1}t\big)$ if~$\varepsilon = -1$. Then the~equalities $e^{\prime}(\varepsilon) = \big(e(\varepsilon),\, t_{e}^{\varepsilon}t\big)$ and~$e^{\prime}(-\varepsilon) =\nolinebreak \big(e(-\varepsilon),\, t\big)$ hold by~the~definition of~$\Gamma^{\prime}$. Therefore, the~constructed mapping of~vertices defines the~desired embedding of~$\Gamma$ into~$\Gamma^{\prime}$. It~is also easy to~see that, in~combination with~the~isomorphisms~$\iota_{v,t}$ ($v \in \mathcal{V}$, $t \in T$), it determines an~embedding of~$\mathcal{G}(\Gamma)$ into~$\mathcal{G}^{\prime}(\Gamma^{\prime})$. Thus, by~Proposition~\ref{p1201}, the~group~$\pi_{1}(\mathcal{G}(\Gamma))$ can be embedded into~the~group~$\pi_{1}(\mathcal{G}^{\prime}(\Gamma^{\prime}))$, as~required.
\end{proof}

\section{Proof of~Theorem~\ref{t2201}}

Given a~class of~groups~$\mathcal{C}$ and~a~group~$X$, we denote by~$\mathcal{C}^{*}(X)$ the~family of~normal subgroups of~$X$ such that $Y \in \mathcal{C}^{*}(X)$ if and~only if $X/Y \in \mathcal{C}$.

\begin{proposition}\label{p1306}
If $\mathcal{C}$ is a~class of~groups closed under~taking subgroups\textup{,} $X$~is a~group\textup{,} $Y$~and~$Z$ are its subgroups\textup{,} and~$Y \in \mathcal{C}^{*}(X)$\textup{,} then $Y \cap Z \in \mathcal{C}^{*}(Z)$.
\end{proposition}

\begin{proof}
Indeed, $Z/Y \cap Z \cong ZY/Y \leqslant X/Y \in \mathcal{C}$ and~$Z/Y \cap Z \in \mathcal{C}$ because $\mathcal{C}$ is closed under~taking subgroups.
\end{proof}

\begin{proposition}\label{p1401}
\textup{\cite[Theorem~1]{Sokolov2015CA}}
If $\mathcal{C}$ is a~class of~groups closed under~taking subgroups\textup{,} then the~following statements are equivalent.

\textup{1.}\hspace{1ex}The~class~$\mathcal{C}$ satisfies the~\emph{Gruenberg condition}\textup{:} for~any group~$X$ and~for~any subnormal series $1 \leqslant Z \leqslant Y \leqslant X$ whose factors $X/Y$ and~$Y/Z$ belong to~$\mathcal{C}$\textup{,} there exists a~subgroup $T \in \mathcal{C}^{*}(X)$ such that $T \leqslant Z$.

\textup{2.}\hspace{1ex}The~class~$\mathcal{C}$ is closed under~taking Cartesian wreath products.
\pagebreak

\textup{3.}\hspace{1ex}The~class~$\mathcal{C}$ is closed under~taking extensions and\textup{,} for~any two groups $X, Y \in \mathcal{C}$\textup{,} contains the~Cartesian product $\prod_{y \in Y}X_{y}$\textup{,} where $X_{y}$ is an~isomorphic copy of~$X$ for~each $y \in Y$.
\end{proposition}

\begin{proposition}\label{p1405}
If $\mathcal{C}$ is a~root class of~groups\textup{,} then the~following statements hold.

\textup{1.}\hspace{1ex}Every free group is residually a~$\mathcal{C}$\nobreakdash-group~\textup{\cite[Theorem~1]{AzarovTieudjo2002PIvSU}}.

\textup{2.}\hspace{1ex}The~free product of~any number of~residually $\mathcal{C}$\nobreakdash-groups is residually a~$\mathcal{C}$\nobreakdash-group \textup{\cite[The\-o\-rem~4.1]{Gruenberg1957PLMS};} \textup{\cite[Theorem~2]{AzarovTieudjo2002PIvSU}}.

\textup{3.}\hspace{1ex}Any extension of~a~residually $\mathcal{C}$\nobreakdash-group by~a~$\mathcal{C}$\nobreakdash-group is again residually a~$\mathcal{C}$\nobreakdash-group\linebreak \textup{\cite[Lemma~1.5]{Gruenberg1957PLMS}}.
\end{proposition}

For~any family of~groups~$\Omega$, we denote by~$\mathcal{P}(\Omega)$ the~class of~groups consisting of~(ordinary) free products, each factor of~which is a~free group or~can be embedded in~a~group from~$\Omega$. It~follows from~Proposition~\ref{p1107} that this class is closed under~taking subgroups. Therefore, if $\Theta$ is a~family of~$\mathcal{P}(\Omega)$\nobreakdash-groups, then $\mathcal{P}(\Theta) \subseteq \mathcal{P}(\Omega)$.

\begin{proposition}\label{p2203}
Suppose that $\mathcal{C}$ is a~root class of~groups and~there exists a~homomorphism~$\sigma$ of~$\pi_{1}(\mathcal{G}(\Gamma))$ onto~a~$\mathcal{C}$\nobreakdash-group that acts injectively on~all edge subgroups~$H_{\varepsilon e}$ \textup{(}$e \in \mathcal{E}$\textup{,} $\varepsilon = \pm 1$\textup{)}. If~$N = \operatorname{ker}\sigma$ and~$\Omega = \{N \cap G_{v} \mid v \in \mathcal{V}\}$\textup{,} then $\pi_{1}(\mathcal{G}(\Gamma))$ is an~extension of~a~$\mathcal{P}(\Omega)$\nobreakdash-group by~a~$\mathcal{C}$\nobreakdash-group.
\end{proposition}

\begin{proof}
Assume first that $\Gamma$ is a~tree, and~consider the~HNN-extension~$\mathbb{E}$ from~Proposition~\ref{p2202}. Obviously, $\mathbb{T} = \pi_{1}(\mathcal{G}(\Gamma))$ is the~quotient group of~$\mathbb{E}$ by~the~normal closure of~the~set $\{t_{e} \mid e \in \mathcal{E}\}$. Let us denote by~$K$ the~preimage of~$N$ under~the~natural homomorphism $\varepsilon\colon \mathbb{E} \to \mathbb{T}$. Then $K \in \mathcal{C}^{*}(\mathbb{E})$ and,~because $\varepsilon$ acts on~every vertex group of~$\mathbb{E}$ as~the~identity mapping, the~equality $K \cap G_{v} = N \cap G_{v}$ holds for~any $v \in \mathcal{V}$. This implies, in~particular, that $K \cap H_{\varepsilon e} = 1$ for~all $e \in \mathcal{E}$, $\varepsilon = \pm 1$. Hence, by~Proposition~\ref{p1106}, $K \in \mathcal{P}\big(\{K \cap \mathbb{P}\}\big)$, where $\mathbb{P}$ is the~base group of~$\mathbb{E}$, i.\,e.,~the~free product of~groups~$G_{v}$~($v \in\nolinebreak \mathcal{V}$). According to~Proposition~\ref{p1107}, $K \cap \mathbb{P} \in \mathcal{P}(\Theta)$, where
$$
\Theta = \big\{(K \cap \mathbb{P}) \cap G_{v} \mid v \in \mathcal{V}\big\}.
$$
Since
$$
\Theta = \big\{K \cap G_{v} \mid v \in \mathcal{V}\big\} = \Omega,
$$
we have $K \cap \mathbb{P} \in \mathcal{P}(\Omega)$ and~$K \in \mathcal{P}(\Omega)$. By~Proposition~\ref{p2202}, $\mathbb{T}$ can be considered a~subgroup of~$\mathbb{E}$ and,~therefore, turns~out to~be an~extension of~$\mathbb{T} \cap K$ by~$\mathbb{T}/\mathbb{T} \cap K$. It~remains to~note that $\mathbb{T} \cap K \in \mathcal{P}(\Omega)$ and~$\mathbb{T}/\mathbb{T} \cap K \in \mathcal{C}$: this follows from~Proposition~\ref{p1306} and~the~fact that the~classes~$\mathcal{P}(\Omega)$ and~$\mathcal{C}$ are closed under~taking subgroups.

So, if $\Gamma$ is a~tree, the~proposition is proved. Assume now that $\Gamma$ is an~arbitrary connected graph and~$\mathcal{T}$ is a~maximal tree in~$\Gamma$ used to~construct a~representation of~$\pi_{1}(\mathcal{G}(\Gamma))$. Since the~restriction~$\sigma_{\mathcal{T}}$ of~$\sigma$ to~the~tree product~$\pi_{1}(\mathcal{G}(\mathcal{T}))$ acts injectively on~all edge subgroups of~this product~and
$$
\big\{\kern-2pt{}\operatorname{ker}\sigma_{\mathcal{T}} \cap G_{v} \mid v \in \mathcal{V}\big\} = \Omega,
$$
it follows from~the~above that $\pi_{1}(\mathcal{G}(\mathcal{T}))$ is an~extension of~some $\mathcal{P}(\Omega)$\nobreakdash-group~$P$ by~a~group from~$\mathcal{C}$. Let us put $U = N \cap \pi_{1}(\mathcal{G}(\mathcal{T}))$ and~$V = U \cap P$. Then $V \in \mathcal{C}^{*}(U) \cap \mathcal{P}(\Omega)$ because the~classes~$\mathcal{P}(\Omega)$ and~$\mathcal{C}$ are closed under~taking subgroups and~Proposition~\ref{p1306} can be applied to~the~subgroups $P \in \mathcal{C}^{*}(\pi_{1}(\mathcal{G}(\mathcal{T})))$~and~$U$.

Since $\pi_{1}(\mathcal{G}(\Gamma))$ is an~HNN-extension of~$\pi_{1}(\mathcal{G}(\mathcal{T}))$, the~subgroup~$N$ splits, by~Proposition~\ref{p1106}, as~the~free product of~a~free group~$F$ and~groups~$X_{i}$ ($i \in \mathcal{I}$) isomorphic to~$U$. Let $\theta\colon N \to U$ be the~surjective homomorphism extending the~isomorphisms $X_{i} \to U$ ($i \in \mathcal{I}$) and~taking~$F$ to~$1$. Let also $M = \operatorname{ker}\theta\delta$, where $\delta\colon U \to U/V$ is a~natural homomorphism. Then $M \in \mathcal{C}^{*}(N)$ and~$M \cap X_{i} \cong V \in \mathcal{P}(\Omega)$ for~all $i \in \mathcal{I}$. By~applying Proposition~\ref{p1107} to~the~free product~$N$, we~get
$$
M \in \mathcal{P}\big(\{M \cap X_{i} \mid i \in \mathcal{I}\}\big).
$$
Hence, $M \in \mathcal{P}(\Omega)$.

Since $M \leqslant N \leqslant \pi_{1}(\mathcal{G}(\Gamma))$ is a~subnormal sequence whose factors belong to~$\mathcal{C}$, it~follows from~Proposition~\ref{p1401} that $M$ contains a~subgroup $L\kern-1pt{} \in\kern-1pt{} \mathcal{C}^{*}(\pi_{1}(\mathcal{G}(\Gamma)))$.\kern-3pt{} Because~the~class~$\mathcal{P}(\Omega)$ is closed under~taking subgroups, we have $L \in \mathcal{P}(\Omega)$. Therefore, $\pi_{1}(\mathcal{G}(\Gamma))$ is an~extension of~the~$\mathcal{P}(\Omega)$\nobreakdash-group~$L$ by~the~$\mathcal{C}$\nobreakdash-group~$\pi_{1}(\mathcal{G}(\Gamma))/L$.
\end{proof}

\begin{proof}[\textbf{\textup{Proof of~Theorem\kern-1pt{}~\ref{t2201}.}}]
\mbox{}\kern-1pt{}Proposition~\ref{p1405} implies that if $\mathcal{C}$ is a~root class of~groups and~a~family~$\Omega$ consists of~residually $\mathcal{C}$\nobreakdash-groups, then any extension of~a~$\mathcal{P}(\Omega)$\nobreakdash-group by~a~$\mathcal{C}$\nobreakdash-group is residually a~$\mathcal{C}$\nobreakdash-group. Therefore, Theorem~\ref{t2201} immediately follows from~Proposition~\ref{p2203}.
\end{proof}

\section{Proof of~Theorem~\ref{t2601}}

\begin{proposition}\label{p2602}
Suppose that $p$ is a~prime number and~$n = p^{l}$ for~some $l \geqslant 1$. Suppose also that $\lambda$ and~$\mu$ are the~bijections of~the~set $M = \{0,\, 1,\, \ldots,\, n-1\}$ defined as~follows\textup{:}{\parfillskip=0pt{}\par}

\vspace*{-12pt}

\begin{align*}
\lambda(i) &= (i+1) \bmod n, \\
\mu(i) &= \begin{cases}
i+1, & i \not\equiv p-1 \pmod p,\\
i-(p-1), & i \equiv p-1 \pmod p.
\end{cases}
\end{align*}

\vspace*{4pt}

\noindent
If~$X$ denotes the~subgroup of~the~group of~bijective mappings of~$M$ generated by~the~elements~$\lambda$ and~$\mu$\textup{,} then $X^{n} = 1$ and\textup{,}~therefore\textup{,} $X$ is a~finite $p$\nobreakdash-group.
\end{proposition}

\begin{proof}
Let $x$ be an~element of~$X$ written as~a~product of~the~generators~$\lambda$ and~$\mu$. Let also $\sigma_{\lambda}(x)$ and~$\sigma_{\mu}(x)$ denote the~sums of~exponents of~$\lambda$ and~$\mu$ in~this product. It~is easy to~see that, for~any $i \in M$,
\vspace*{4pt}
\begin{align*}
\lambda(i) &= \begin{cases}
\mu(i), & i \not\equiv p-1 \pmod p,\\
(\mu\lambda^{p})(i), & i \equiv p-1 \pmod p,
\end{cases}
\\
\lambda^{-1}(i) &= \begin{cases}
\mu^{-1}(i), & \hspace*{2.5pt}i \not\equiv 0 \pmod p,\\
(\lambda^{-p}\mu^{-1})(i), & \hspace*{2.5pt}i \equiv 0 \pmod p.
\end{cases}
\end{align*}

It follows that, for~each $i \in M$, there exist an~element $y_{i} \in \operatorname{sgp}\{\lambda^{p}, \mu\}$ and~a~number $k_{i} \in \mathbb{Z}$ such that
$$
(x^{p})(i) = y_{i}(i),\quad 
\sigma_{\mu}(y_{i}) = p\big(\sigma_{\lambda}(x) + \sigma_{\mu}(x)\big),\quad 
\sigma_{\lambda}(y_{i}) = pk_{i}
$$
(here, as~above, $\sigma_{\lambda}(y_{i})$ and~$\sigma_{\mu}(y_{i})$ denote the~sums of~exponents of~$\lambda$ and~$\mu$ in~the~fixed representation of~$y_{i}$ as~a~product of~generators). Since $\mu^{p} = [\lambda^{p}, \mu] = 1$, the~equality $y_{i} = \lambda^{pk_{i}}$
holds. Therefore,
$$
(x^{p})(i) = (\lambda^{pk_{i}})(i) = (i+pk_{i}) \bmod n.
$$

If $j = (i+pr) \bmod n$ for~some $r \in \mathbb{Z}$, then
$$
(x^{p})(j) = (\lambda^{pr}x^{p})(i) = (x^{p}\lambda^{pr})(i) = \big(i+p(k_{i}+r)\big) \bmod n.
$$
Using these relations and~obvious induction, we get that, for~any $s \geqslant 1$, the~equality
$$
(x^{p})^{s}(i) = (i+psk_{i}) \bmod n
$$
holds. Therefore, $(x^{p})^{n/p}(i) = i$, and~since $x$ and~$i$ are chosen arbitrarily, $X^{n} = 1$.
\end{proof}

\begin{proposition}\label{p2603}
Suppose that $p$\textup{,} $n$\textup{,} $M$\textup{,} $\lambda$\textup{,} and~$\mu$ are defined in~the~same way as~in~Proposition~\textup{\ref{p2602}}. Suppose also that $C_{0}$\textup{,}~\ldots\textup{,}~$C_{n-1}$ are cyclic groups of~order~$p$ with~generators $c_{0}$\textup{,}~\ldots\textup{,}~$c_{n-1}$ respectively\textup{;} $H_{n}$~is the~direct product of~the~groups~$C_{i}$ \textup{(}$i \in M$\textup{);} $\alpha_{n}$~and~$\beta_{n}$ are the~automorphisms of~$H_{n}$ acting according to~the~rule\textup{:}
$$
c_{i}\alpha_{n} = c_{\mu(i)},\quad 
c_{i}\beta_{n} = c_{(\lambda^{-1}\mu\lambda)(i)}\quad 
(i \in M);
$$
$A_{n}$~and~$B_{n}$ are the~split extensions of~$H_{n}$ by~the~cyclic groups~$\langle \alpha_{n} \rangle$ and~$\langle \beta_{n} \rangle$\textup{,} respectively. Then the~generalized free product~$P_{n}$ of~the~groups~$A_{n}$ and~$B_{n}$ with~the~amalgamated subgroup~$H_{n}$ is residually $p$\nobreakdash-finite.
\end{proposition}

\begin{proof}
Since $H_{n}$ is normal in~$A_{n}$ and~$B_{n}$, we can consider the~group~$\operatorname{Aut}_{P_{n}}(H_{n})$ consisting of~the~restrictions on~$H_{n}$ of~all inner automorphisms of~$P_{n}$. Obviously, $\operatorname{Aut}_{P_{n}}(H_{n})$ is~generated by~$\alpha_{n}$ and~$\beta_{n}$ and~is isomorphic to~the~subgroup of~the~group~$X$ from~Proposition~\ref{p2602} generated by~the~bijections~$\mu$ and~$\lambda^{-1}\mu\lambda$. It~follows that $\operatorname{Aut}_{P_{n}}(H_{n})$ is a~finite $p$\nobreakdash-group. Since $A_{n}$ and~$B_{n}$ are also finite $p$\nobreakdash-groups, $P_{n}$~is residually $p$\nobreakdash-finite by~\cite[Corollary~2]{Higman1964JA}.
\end{proof}

\begin{proposition}\label{p2604}
Suppose that $p$ is a~prime number\textup{,} $\lambda_{\infty}$~and~$\mu_{\infty}$ are the~bijections of~$\mathbb{Z}$ defined as~follows\textup{:}
\begin{align*}
\lambda_{\infty}(i) &= i+1,\\
\mu_{\infty}(i) &= \begin{cases}
i+1, & i \not\equiv p-1 \pmod p,\\
i-(p-1), & i \equiv p-1 \pmod p.
\end{cases}
\end{align*}
Suppose also that $C_{i}$ is a~cyclic group of~order~$p$ with~a~generator~$c_{i}$ for~any $i \in \mathbb{Z}$\textup{;} $H_{\infty}$~is the~direct product of~the~groups~$C_{i}$ \textup{(}$i \in \mathbb{Z}$\textup{);} $\alpha_{\infty}$~and~$\beta_{\infty}$ are the~automorphisms of~$H_{\infty}$ acting according to~the~rule\textup{:}
$$
c_{i}^{\vphantom{1}}\alpha_{\infty}^{\vphantom{1}} = c_{\mu_{\infty}(i)}^{\vphantom{1}},\quad 
c_{i}^{\vphantom{1}}\beta_{\infty}^{\vphantom{1}} = c_{(\lambda_{\infty}^{-1}\mu_{\infty}\lambda_{\infty})(i)}\quad 
(i \in \mathbb{Z});
$$
$A_{\infty}$~and~$B_{\infty}$ are the~split extensions of~$H_{\infty}$ by~the~cyclic groups~$\langle \alpha_{\infty} \rangle$ and~$\langle \beta_{\infty} \rangle$\textup{,} respectively. Then the~following statements hold.

\textup{1.}\hspace{1ex}The~generalized free product~$P_{\infty}$ of~the~groups~$A_{\infty}$ and~$B_{\infty}$ with~the~amalgamated subgroup~$H_{\infty}$ is residually $p$\nobreakdash-finite.

\textup{2.}\hspace{1ex}The~automorphism~$\alpha_{\infty}^{-1}\beta_{\infty}^{\vphantom{1}}$ has an~infinite order.
\end{proposition}

\begin{proof}
1.\hspace{1ex}Let us take an~element $x \in P_{\infty} \setminus \{1\}$ and~find a~homomorphism of~$P_{\infty}$ onto~a~finite $p$\nobreakdash-group that maps~$x$ to~a~non-trivial element.

The~subgroup~$H_{\infty}$ is normal in~$P_{\infty}$, and~the~quotient group~$P_{\infty}/H_{\infty}$ splits as~the~(ordinary) free product of~two finite $p$\nobreakdash-groups~$\langle \alpha_{\infty} \rangle$ and~$\langle \beta_{\infty} \rangle$. Therefore, if $x \notin H_{\infty}$, then the~natural homomorphism $P_{\infty} \to P_{\infty}/H_{\infty}$ can be extended to~the~desired one by~Proposition~\ref{p1405}.

If $x \in H_{\infty}$, then there exists a~$p$\nobreakdash-number~$n$ such that
$$
x\kern3pt{} \in\kern-7pt{} \prod_{-n/2 \leqslant i < n/2\phantom{-}}\kern-4pt{} C_{i}.
$$
Let us consider the~mapping of~the~generators of~$P_{\infty}$ into~the~group~$P_{n}$ from~Proposition~\ref{p2603} acting according to~the~rule
$$
\alpha_{\infty} \mapsto \alpha_{n},\quad 
\beta_{\infty} \mapsto \beta_{n},\quad 
c_{i} \mapsto c_{i \bmod n}.
$$
Since $p \mid n$, this mapping defines a~homomorphism, which we denote by~$\sigma$. It~is clear that $x \notin \operatorname{ker}\sigma$ due~to~the~choice of~$n$. Therefore, $\sigma$ can be extended to~the~desired mapping.{\parfillskip=0pt{}\par}

2.\hspace{1ex}It can be directly verified that if $i \in \mathbb{Z}$ and~$i \equiv 0 \pmod p$, then $[\mu_{\infty}, \lambda_{\infty}](i) = i+p$. Hence, $c_{i}^{\vphantom{1}}(\alpha_{\infty}^{-1}\beta_{\infty}^{\vphantom{1}}) = c_{i+p}^{\vphantom{1}}$ and,~therefore, the~order of~$\alpha_{\infty}^{-1}\beta_{\infty}^{\vphantom{1}}$ is infinite.
\end{proof}

\begin{proof}[\textup{\textbf{Proof of~Theorem~\ref{t2601}.}}]
Since $\mathcal{C}$ includes at~least one infinite periodic group and~is closed under~taking subgroups, extensions, and~Cartesian powers, there exists a~prime number~$p$ (assumed to~be fixed below) such that $\mathcal{C}$ contains a~cyclic group of~order~$p$, the~Cartesian product of~an~infinite number of~such groups, and~the~class of~all finite $p$\nobreakdash-groups. Hence, the~groups~$H_{\infty}$, $A_{\infty}$, and~$B_{\infty}$ from~Proposition~\ref{p2604} belong to~$\mathcal{C}$. Let us define a~graph of~groups~$\mathcal{G}(\Gamma)$ as~follows.

If $\Gamma$ has only one vertex, then we assign to~it the~group~$H_{\infty}$, while to~each edge $e \in \mathcal{E}$ the~same group~$H_{\infty}$ and~the~homomorphisms $\varphi_{+e}^{\vphantom{1}} = \mathrm{id}_{H_{\infty}}^{\vphantom{1}}$, $\varphi_{-e}^{\vphantom{1}} = \alpha_{\infty}^{-1}\beta_{\infty}^{\vphantom{1}}$, where $\alpha_{\infty}^{\vphantom{1}}$ and~$\beta_{\infty}^{\vphantom{1}}$ are the~automorphisms from~Proposition~\ref{p2604}. Otherwise, we choose some edge~$f$ of~$\Gamma$ that is not a~loop and~associate the~vertex~$f(1)$ with~the~group~$A_{\infty}$, all other vertices with~the~group~$B_{\infty}$, and~all the~edges with~the~group~$H_{\infty}$ and~its identity embeddings in~$A_{\infty}$ and~$B_{\infty}$.

It is easy to~see that, in~both cases and~for~any choice of~the~maximal tree~$\mathcal{T}$ in~$\Gamma$, all edge subgroups of~the~group $\pi_{1}(\mathcal{G}(\Gamma)) = \pi_{1}(\mathcal{G}(\Gamma), \mathcal{T})$ coincide, are normal in~$\pi_{1}(\mathcal{G}(\Gamma))$, and~are equal to~$H_{\infty}$ (we again denote this unique subgroup by~$H_{\infty}$). To~prove Statement~3, let us fix an~element $g \in \pi_{1}(\mathcal{G}(\Gamma)) \setminus \{1\}$ and~find a~homomorphism of~$\pi_{1}(\mathcal{G}(\Gamma))$ onto~a~$\mathcal{C}$\nobreakdash-group taking~$g$ to~a~non-trivial element.

The~quotient group $\pi_{1}(\mathcal{G}(\Gamma))/H_{\infty}$ is either a~free group whose basis is the~stable letters of~$\pi_{1}(\mathcal{G}(\Gamma))$, or~the~(ordinary) free product of~this group and~finite $p$\nobreakdash-groups, which belong to~$\mathcal{C}$ and~are isomorphic to~$A_{\infty}/H_{\infty}$ or,~what is the~same,~$B_{\infty}/H_{\infty}$. Hence, it is residually a~$\mathcal{C}$\nobreakdash-group by~Proposition~\ref{p1405}, and~if $g \notin H_{\infty}$, then the~natural homomorphism\linebreak $\pi_{1}(\mathcal{G}(\Gamma)) \to \pi_{1}(\mathcal{G}(\Gamma))/H_{\infty}$ can be extended to~the~desired~one.

Let $g \in H_{\infty}$. Consider the~mapping of~the~generators of~$\pi_{1}(\mathcal{G}(\Gamma))$ into~the~group~$P_{\infty}$ from~Proposition~\ref{p2604} that acts on~the~generators of~all vertex groups as~the~identity mapping and~takes the~symbols~$t_{e}$ either to~the~element~$\alpha_{\infty}^{-1}\beta_{\infty}^{\vphantom{1}}$ (if~$\Gamma$ has one vertex), or~to~$1$ (otherwise). It~is easy to~see that this map defines a~homomorphism, which is injective on~$H_{\infty}$. The~group~$P_{\infty}$ is residually $p$\nobreakdash-finite by~Proposition~\ref{p2604}, and~the~class~$\mathcal{C}$ contains all finite $p$\nobreakdash-groups. Hence, the~constructed homomorphism can again be extended to~the~desired one.

Let us now turn to~the~proof of~Statement~4 and~show that there exists an~element $\gamma \in \pi_{1}(\mathcal{G}(\Gamma))$ such that the~conjugation by~$\gamma$ acts on~$H_{\infty}$ as~the~automorphism~$\alpha_{\infty}^{-1}\beta_{\infty}^{\vphantom{1}}$.{\parfillskip=0pt{}\par}

Indeed, if $\Gamma$ has one vertex, then for~some $e \in \mathcal{E}$, the~element~$t_{e}$ can be taken as~$\gamma$. Otherwise, $\gamma = \alpha_{\infty}^{-1}\beta_{\infty}^{\vphantom{1}}$ or~$\gamma = (t_{f}^{-1}\alpha_{\infty}^{\vphantom{1}}t_{f}^{\vphantom{1}})^{-1}_{\vphantom{f}}\beta_{\infty}^{\vphantom{1}}$; it~depends on~whether or~not $\mathcal{T}$ contains the~edge~$f$ chosen above (here $\alpha_{\infty}$ and~$\beta_{\infty}$ are the~elements of~the~groups~$A_{\infty}$ and~$B_{\infty}$ associated with~the~vertices~$f(1)$ and~$f(-1)$).

If $\sigma$ is a~homomorphism of~$\pi_{1}(\mathcal{G}(\Gamma))$ onto~a~periodic group, then $(\gamma\sigma)^{k} = 1$ for~some $k \geqslant 1$. Since, by~Proposition~\ref{p2604}, the~automorphism~$\alpha_{\infty}^{-1}\beta_{\infty}^{\vphantom{1}}$ has an~infinite order, there exists an~element $h \in H_{\infty}$ satisfying the~relations
$$
h \ne h(\alpha_{\infty}^{-1}\beta_{\infty}^{\vphantom{1}})^{k} = \gamma^{-k}h\gamma^{k}.
$$
Therefore,
$$
1 \ne [h, \gamma^{k}] \in \operatorname{ker}\sigma \cap H_{\infty} = \operatorname{ker}\sigma \cap H_{\pm e}
$$
for all $e \in \mathcal{E}$.
\end{proof}

\end{document}